\documentclass{amsart}

\usepackage{amsmath,amsxtra,amsthm,amssymb,amsfonts}

\theoremstyle{plain}

\newtheorem{thm}{Theorem}[section]

\newtheorem{prop}[thm]{Proposition}
\newtheorem{lem}[thm]{Lemma}

\theoremstyle{definition}

\newtheorem{rem}[thm]{Remark}
\newtheorem{exa}[thm]{Example}

\def\nat{\mathbb{N}}
\def\rls{\mathbb{R}}
\def\exrls{(-\infty,\infty]}

\def\wto{\stackrel{w}{\to}}
\def\ol{\overline}
\def\clco{\operatorname{\ol{co}}}
\def\argmin{\operatornamewithlimits{\arg\min}}

\def\dom{\operatorname{dom}}
\def\cldom{\operatorname{\ol{dom}}}

\def\res{\restriction}

\begin{document}
\title[The proximal point algorithm]{The proximal point algorithm in metric spaces}
\author[M. Ba\v{c}\'ak]{Miroslav Ba\v{c}\'ak}
\date{\today}
\thanks{This paper was accepted to Israel journal of mathematics. The final version may differ.}
\address{Miroslav Ba\v{c}\'ak, Max Planck Institute, Inselstrasse 22, 04 103 Leipzig, Germany}
\email{bacak@mis.mpg.de}

\begin{abstract} The proximal point algorithm, which is a well-known tool for finding minima of convex functions, is generalized from the classical Hilbert space framework into a nonlinear setting, namely, geodesic metric spaces of nonpositive curvature. We prove that the sequence generated by the proximal point algorithm weakly converges to a minimizer, and also discuss a related question: convergence of the gradient flow.
\end{abstract}

\maketitle

\section{Introduction}

The proximal point algorithm (PPA) is a method for finding a minimizer of a~convex lower semicontinuous (shortly, lsc) function defined on a Hilbert space. Its origin goes back to Martinet, Rockafellar, and Br\'ezis\&Lions \cite{brezis,martinet,rocka}. The algorithm has since become extremely popular among practitioners in optimization, and also offered many challenging mathematical problems. For instance, Rockafellar's 1976 question \cite{rocka} as to whether or not the PPA always converges strongly was settled (in the negative) as late as 1990s by G\"uler \cite{guler}. The literature on the subject has become too extensive to be even partially listed here. We believe the interested reader will easily find further information on this field.

Gradually, many of the algorithms for solving optimization problems have been generalized from linear spaces (Euclidean, Hilbert, Banach) into differentiable manifolds. In particular, the proximal point algorithm in the context of Riemannian manifold (of nonpositive sectional curvature) was studied in \cite{bento,feroli,sevilla,papa}. We continue along these lines and introduce the PPA into geodesic metric spaces of~nonpositive curvature, so-called CAT(0) spaces. Since minimizers of convex lsc functionals in these spaces play an important role in analysis and geometry (see, for instance, Sections~\ref{sec:hoc} and~\ref{subsec:energy}), we dare to believe that the PPA will prove useful. Also, we should like to mention Zaslavski's very recent paper \cite{zasl} with a different approach to the PPA in metric spaces.

\emph{The aim of this paper is to introduce the proximal point algorithm into metric spaces of nonpositive curvature and show weak convergence of this algorithm.}

There are of course natural obstacles one has to overcome in CAT(0) spaces. Unlike Riemannian manifolds, CAT(0) spaces do not come equipped with a Riemannian metric, and, probably relatedly, we do not have a notion of a subgradient of a~convex function. The proof of weak convergence of the PPA in Hilbert spaces, on the other hand, does use both the inner product and the convex subgradient~\cite{brezis,guler,martinet,rocka}, and therefore we cannot simply translate the existing proof into the context of metric spaces. Furthermore, a general CAT(0) space is not locally compact, which does not in general allow to prove strong convergence of the PPA, and forces us to manage with the weak convergence.

The results of the present paper can be of an interest also in Hilbert spaces, as they show that the PPA as well as the gradient flow semigroup are purely metric objects in spite of their linear origins.

\subsection{Proximal point algorithm}
Let $H$ be a Hilbert space and $f:H\to\exrls$ be a convex lsc function which attains its minimum on $H.$ The proximal point algorithm seeks a minimizer of $f$ by successive approximations
\begin{equation}
 x_n=\argmin_{y\in H}\left[f(y)+\frac1{2\lambda_n}\|y-x_{n-1}\|^2\right], \qquad n\in\nat,
\end{equation}
where $x_0\in H$ is a given starting point, and $\lambda_n>0$ for all $n\in\nat.$ The sequence $(x_n)$ is known to converge weakly to a minimizer of $f,$ provided $\sum_1^\infty\lambda_n=\infty,$ see \cite{brezis,guler,martinet,rocka}. A natural question, posed by Rockafellar in \cite{rocka}, whether this convergence can be improved to strong was answered in the negative by G\"uler \cite[Corollary 5.1]{guler}. In other words, weak convergence is the best we can achieve without additional assumptions. It is worth mentioning that counterexamples to strong convergence of the PPA are still very rare \cite{prox,kopecka}.

\subsection{CAT(0) spaces} \label{sec:hoc} Geodesic metric spaces of nonpositive curvature in the sense of Alexandrov, that is, CAT(0) spaces in Gromov's terminology, include Hilbert spaces, $\rls$-tree, Euclidean Bruhat-Tits buildings, complete simply connected Riemannian manifolds of nonpositive sectional curvature, and many other important spaces included in none of the above classes.

There are several equivalent conditions for a geodesic metric space $(X,d)$ to be CAT(0), one of them is the following inequality, which is to be satisfied for any $x\in X,$ any geodesic $\gamma:[a,b]\to X,$ and any $t\in[0,1]:$
\begin{equation} \label{eq:cat}
 d\left(x,\gamma(t)\right)^2\leq (1-t)d\left(x,\gamma(a)\right)^2+td\left(x,\gamma(b)\right)^2-t(1-t)d\left(\gamma(a),\gamma(b)\right)^2.
\end{equation}

Convex functions on CAT(0) spaces are our principal object of interest in this paper. Recall that a function $f:C\to\exrls,$ defined on a convex subset $C$ of a CAT(0) space, is \emph{convex} if, for any geodesic $\gamma:[0,1]\to C,$ the function $f\circ\gamma$ is convex. Here we collect several important examples \cite{book}.

\begin{exa}[Distance functions]
The function
\begin{equation}\label{eq:dist} x\mapsto d\left(x,x_0\right),  \qquad x\in X,\end{equation}
where $x_0$ is a fixed point of $X,$ is convex and continuous. The square of this function is even \emph{strictly} convex. More generally, the distance function $d_C$ to a closed convex subset $C\subset X,$ defined as
$$d_C(x)=\inf_{c\in C} d(x,c),  \qquad x\in X,$$
is convex and $1$-Lipschitz \cite[Proposition 2.4, p.176]{book}.
\end{exa}

\begin{exa}[Displacement functions]
Let $T:X\to X$ be an isometry. The \emph{displacement function} of $T$ is the function $\delta_T:X\to[0,\infty)$ defined by
$$\delta_T(x)=d(x,Tx),$$
for all $x\in X.$ It is convex and Lipschitz \cite[Definition II.6.1]{book}.
\end{exa}

\begin{exa}[Busemann functions]
Let $c:[0,\infty)\to X$ be a geodesic ray. The function $b_c:X\to\rls$ defined by
$$ b_c(x)= \lim_{t\to\infty} \left[d\left(x,c(t) \right) -t \right],  \qquad x\in X,$$
is called the \emph{Busemann function} associated to the ray $c,$ see \cite[Definition II.8.7]{book}. Busemann functions are convex and $1$-Lipschitz. Concrete examples of Busemann functions are given in \cite[p. 273]{book}. Another explicit example of a Busemann function in the CAT(0) space of positive definite $n\times n$ matrices with real entries is found in~\cite[Proposition 10.69]{book}. The sublevel sets of Busemann functions are called \emph{horoballs} and carry a lot of information about the geometry of the space in question, see \cite{book} and the references therein.
\end{exa}

The energy functional is another important instance of a convex function on a CAT(0) space, see \cite[Chapter 7]{jost1}, or more generally in \cite[Chapter 4]{jost2}. Minimizers of the energy functional are called \emph{harmonic maps,} and are of an immense importance in both geometry and analysis.

\subsection{Resolvents and semigroups} \label{subsec:energy}

Let $(X,d)$ be a complete CAT(0) space, and $f:X\to\exrls$ be lsc convex. For $\lambda>0,$ define the \emph{Moreau-Yosida resolvent} of $f$ as
$$J_\lambda(x)=\argmin_{y\in X} \left[f(y)+\frac1{2\lambda}d(y,x)^2\right],$$
and put $J_0(x)=x,$ for all $x\in X.$ This definition in metric spaces with no linear structure first appeared in \cite{jost-o}. The mapping $J_\lambda$ is well defined for all $\lambda\geq0,$ see \cite[Lemma 2]{jost-o} and \cite[Theorem 1.8]{mayer}.

The Moreau-Yosida resolvents are essential in the proof of existence of harmonic maps. Indeed, the energy functional is convex and lsc on a suitable CAT(0) space $Y$ of $\mathcal{L}^2$-mappings, and $J_\lambda(y),$ with an arbitrary $y\in Y,$ is shown to strongly converge to a minimizer of the energy functional (to a harmonic map), as~$\lambda\to\infty.$ For the details, see \cite{jost1,jost-ch,jost2,jost-o}.

In spite of the significance of the convergence of $J_\lambda,$ it is more desirable to establish convergence of the corresponding (gradient flow) semigroup $\left(T_\lambda\right)_{\lambda\geq0},$ which is given as
\begin{equation} \label{eq:flow}
T_\lambda x=\lim_{n\to\infty} \left(J_{\frac{\lambda}n}\right)^n (x),\qquad x\in \cldom f.
\end{equation}
The limit in \eqref{eq:flow} is uniform with respect to $\lambda$ on bounded subintervals of $[0,\infty),$ and $\left(T_\lambda\right)_{\lambda\geq0}$ is a strongly continuous semigroup of nonexpansive mappings, see \cite[Theorem 1.3.13]{jost-ch}, and \cite[Theorem 1.13]{mayer}. Unfortunately, the semigroup $\left(T_\lambda\right)$ convergences only weakly (see Theorem~\ref{thm:flow} below and the subsequent discussion).

It is well known that the PPA is a discrete version of the gradient flow semigroup~\cite{guler,mayer}.

\subsection{Main results}
Let $(X,d)$ be a complete CAT(0) space and $f:X\to\exrls$ be a lsc convex function. The \emph{proximal point algorithm} starting at a point $x_0\in X$ generates in the $n$-th step, $n\in\nat,$ the point 
\begin{equation} \label{eq:ppa}
 x_n=\argmin_{y\in X}\left[f(y)+\frac1{2\lambda_n}d\left(y,x_{n-1}\right)^2\right].
\end{equation}
Recall that $x_n$ is well-defined. By G\"uler's result of \cite{guler}, only weak convergence of $(x_n)$ to a minimizer of $f$ can be expected in general.

The main results of the present paper are the following two theorems.
\begin{thm} \label{thm:ppa}
Let $(X,d)$ be a complete CAT(0) space, and $f:X\to\exrls$ be a~convex lsc function. Suppose that $f$ has a minimizer, that is, there exists a point $c\in X$ such that
$$f(c)=\inf_{x\in X} f(x).$$
Then, for an arbitrary starting point $x_0\in X,$ and a sequence of positive reals $\left(\lambda_n\right)$ such that $\sum_1^\infty\lambda_n=\infty,$ the sequence $(x_n)\subset X$ defined by \eqref{eq:ppa} weakly converges to a minimizer of $f.$
\end{thm}
The proof of Theorem~\ref{thm:ppa} is given in Section~\ref{sec:proof}. It is based on Fej\'er monotonicity, whereas the classical Hilbert space proofs do not use this feature \cite{brezis,guler,martinet,rocka}. However, it was later observed by Combettes that the PPA sequence is Fej\'er monotone \cite{combettes}.

We also study an object which is closely related to the PPA, namely, the gradient flow, and obtain the following result.
\begin{thm}\label{thm:flow}
Let $X$ be a complete CAT(0) space, and $f:X\to\exrls$ be lsc convex. Assume that $f$ attains its minimum on $X.$ Then, given a starting point $x\in\cldom f,$ the gradient flow $T_\lambda x$ defined in \eqref{eq:flow} weakly converges to a minimizer of $f,$ as $\lambda\to\infty.$
\end{thm}

In \cite[p. 24]{jost-ch}, the author discusses the following problem. If there is a sequence $\left(\lambda_n\right)\subset(0,\infty)$ such that $\lambda_n\to\infty,$ and that the sequence $\left(T_{\lambda_n}x\right)_n$ is bounded, is it then the case that $\left(T_\lambda x\right)$ converges (strongly) to a minimizer of $f,$ as $\lambda\to\infty$? The answer is no: by Fej\'er monotonicity we know that the sequence $\left(T_{\lambda_n}x\right)_n$ is bounded, see Proposition~\ref{prop:fejer}\eqref{item:iii} below, however Baillon's example \cite{baillon} shows that there is a semigroup which converges weakly to a minimizer, but fails to converge strongly. For the details, see the proof of Theorem~\ref{thm:flow} in Section~\ref{sec:proof}.

We end the Introduction with the following two remarks.
\begin{rem}[Rate of the convergence]
 As we shall see in the proof of Theorem~\ref{thm:ppa}, more precisely in~\eqref{eq:rate}, for any $n\in\nat,$ we have
$$f(x_n)-\inf f\leq\frac{K}{\sum_{k=1}^n\lambda_k},$$
where $K$ is a positive constant. In other words, rate of the weak convergence is
$$f(x_n)-\inf f=O\left(\frac1{\sum_{k=1}^n\lambda_k}\right).$$
\end{rem}
\begin{rem}[Strong convergence] \label{rem:uc} In order to ensure strong convergence in Theorems~\ref{thm:ppa} and~\ref{thm:flow}, we need to impose additional assumptions on the data. For instance, it would be sufficient to require the underlying space to be locally compact. Or, we may require the function $f$ to be uniformly convex on bounded subsets of its domain, see \cite[Theorem 27.1(iii)]{baucom} and \cite[Lemma 1.7]{mayer}. Recall the definition. Let $\phi:[0,\infty)\to[0,\infty]$ be a non-decreasing function vanishing only at $0.$ A function $h:X\to\exrls$ is \emph{uniformly convex} on a set $A\subset\dom f$ with modulus $\phi$ if,
$$ h\left(\alpha x+(1-\alpha)y\right)+\alpha(1-\alpha)\phi\left(d(x,y)\right)\leq\alpha h(x)+(1-\alpha)h(y),$$
for any $x,y\in A$ and any $\alpha\in[0,1].$
We refer the reader to Remark~\ref{rem:uc} for the proof that uniform convexity of the function implies strong convergence in Theorems~\ref{thm:ppa} and~\ref{thm:flow}.

\end{rem}

\subsection*{Acknowledgment} I am grateful to J\"urgen Jost for his very valuable comments.

\section{Preliminaries}
We first recall basic notation concerning CAT(0) spaces. For further details on the subject, the reader is referred to \cite{book}. Let $(X,d)$ be a CAT(0) space. Having two points $x,y\in X,$ we denote the geodesic segment from $x$ to $y$ by $[x,y].$ We usually do not distinguish between a geodesic and its geodesic segment, as no confusion can arise. A set $C\subset X$ is \emph{convex} if $x,y\in C$ implies $[x,y]\subset C.$ For a point $z\in[x,y],$ we write $z=tx+(1-t)y,$ where $t=d(z,y)/d(x,y).$

Given $x,y,z\in X,$ the symbol $\alpha(y,x,z)$ denotes the (Alexandrov) angle between the geodesics $[x,y]$ and $[x,z]$. The corresponding angle in the comparison triangle is denoted $\alpha'(y,x,z).$

\subsection{Metric projections}\label{subsec:proj}
For any metric space $(X,d)$ and $C\subset X,$ define the \emph{distance function} by
$$d_C(x)=\inf_{c\in C} d(x,c),\quad x\in X.$$
Note that the function $d_C$ is convex and continuous provided $X$ is CAT(0) and $C$ is convex and complete \cite[Cor. 2.5, p.178]{book}.
\begin{prop}
Let $(X,d)$ be a CAT(0) space and $C\subset X$ be complete and convex. Then:
\begin{enumerate}
\item For every $x\in X,$ there exists a unique point $P_C(x)\in C$ such that
$$d\left(x,P_C(x)\right)=d_C(x).$$
\item If $y\in\left[x,P_C(x)\right],$ then $P_C(x)=P_C(y).$
\item If $x\in X\setminus C$ and $y\in C$ such that $P_C(x)\neq y,$ then $\alpha\left(x,P_C(x),y\right)\geq\frac\pi2.$
\item The mapping $P_C$ is a non-expansive retraction from $X$ onto $C.$
\end{enumerate}
\label{prop:proj}
\end{prop}
\begin{proof} See \cite[Proposition 2.4, p.176]{book}. \end{proof}

The mapping $P_C$ is called the \emph{metric projection} onto $C.$

\subsection{Weak convergence} \label{subsec:weak}
A notion of weak convergence in CAT(0) spaces was first introduced by J\"urgen Jost in \cite[Definition 2.7]{jost-p}. Sosov later defined his $\psi$- and $\phi$-convergences, both generalizing Hilbert space weak convergence into geodesic metric spaces \cite{sos}. Recently, Kirk and Panyanak extended Lim's $\Delta$-convergence \cite{lim} into CAT(0) spaces \cite{kp} and finally, Esp\'inola and Fern\'andez-Le\'on \cite{efl} modified Sosov's $\phi$-convergence to obtain an equivalent formulation of $\Delta$-convergence in CAT(0) spaces. This is, however, exactly the original weak convergence due to Jost~\cite{jost-p}.

Let $X$ be a complete CAT(0) space. Suppose $(x_n)\subset X$ is a bounded sequence and define its \emph{asymptotic radius} about a given point $x\in X$ as
$$r(x_n,x)=\limsup_{n\to\infty} d(x_n,x),$$
and \emph{asymptotic radius} as
$$r(x_n)=\inf_{x\in X} r(x_n,x).$$
Further, we say that a point $x\in X$ is the \emph{asymptotic center} of $(x_n)$ if
$$ r(x_n,x)=r(x_n).$$
Since $X$ is a complete CAT(0) space we know that the asymptotic center of $(x_n)$ exists and is unique \cite[Proposition 7]{dks}.

We shall say that $(x_n)\subset X$ \emph{weakly converges} to a point $x\in X$ if $x$ is the asymptotic center of each subsequence of $(x_n).$ We use the notation $x_n\wto x.$ Clearly, if $x_n\to x,$ then $x_n\wto x.$

If there is a subsequence $(x_{n_k})$ of $(x_n)$ such that $x_{n_k}\wto z$ for some $z\in X,$ we say that $z$ is a \emph{weak cluster point} of the sequence $(x_n).$ Each bounded sequence has a weak cluster point, see \cite[Theorem 2.1]{jost-p}, or \cite[p. 3690]{kp}.

\begin{prop}
A bounded sequence $(x_n)\subset X$ weakly converges to a point $x\in X$ if and only if, for any geodesic $\gamma:[0,1]\to X$ with $x\in\gamma,$ we have
$$d\left(x,P_\gamma(x_n)\right)\to 0, \quad \text{as } n\to\infty.$$
\label{prop:equiv}
\end{prop}
\begin{proof} See \cite[Proposition 5.2]{efl}. \end{proof}

We shall say that a function $f:X\to\exrls$ is \emph{weakly lsc} at a given point $x\in X$ if
$$\liminf_{n\to\infty} f(x_n)\geq f(x)$$
for each sequence $x_n \wto x.$

It is easy to verify that in Hilbert spaces the classical weak convergence coincides with the weak convergence defined above.

\subsection{Fej\'er monotone sequences}
A sequence $(x_n)\subset X$ is \emph{Fej\'er monotone} with respect to $C$ if, for any $c\in C,$
$$ d(x_{n+1},c) \leq d(x_n,c), \qquad n\in\nat.$$

\begin{prop}
Let $(x_n)\subset X$ be a Fej\'er monotone sequence with respect to $C.$ Then:
\begin{enumerate}
\item $(x_n)$ is bounded, \label{item:i}
\item $d_C(x_{n+1}) \leq d_C(x_n)$ for each $n\in\nat.$ \label{item:ii}
\item $(x_n)$ weakly converges to some $x\in C$ if and only if all weak cluster points of $(x_n)$ belong to $C.$ \label{item:iii}
\item $(x_n)$ converges to some $x\in C$ if and only if $d(x_n,C)\to 0.$ \label{item:iv}
\end{enumerate}
\label{prop:fejer}
\end{prop}
\begin{proof}
 See \cite[Proposition 3.3]{mb}.
\end{proof}

\section{Proofs of the theorems} \label{sec:proof}

We begin with a useful lemma, whose proof follows easily from the fact that a~closed convex subset of a complete CAT(0) space is (sequentially) weakly closed \cite[Lemma 3.1]{mb}. The symbol $\clco A$ stands for the closed convex hull of a subset $A$ of a CAT(0) space $X,$ that is, the smallest closed convex subset of $X$ containing $A.$
\begin{lem} \label{lem:lsc}
Let $X$ be a complete CAT(0) space. If $f:X\to\exrls$ a lsc convex function, then it is weakly lsc.
\end{lem}
\begin{proof}
By contradiction. Let $(x_n)\subset X, x\in X$ and $x_n\wto x.$ Suppose that
$$\liminf_{n\to\infty} f(x_n) < f(x).$$
That is, there exist a subsequence $(x_{n_k}),$ index $k_0\in\nat,$ and $\delta>0$ such that $ f(x_{n_k}) < f(x)-\delta$
for all $k>k_0.$ By lower semicontinuity and convexity of $f,$ we get
$$ f(y) \leq f(x)-\delta$$
for all $y\in \clco \{x_{n_k}:k>k_0\}.$ But this, through \cite[Lemma 3.1]{mb}, yields a contradiction to $x_n\wto x.$
\end{proof}

Now we give the proof of Theorem~\ref{thm:ppa}.

\begin{proof}[Proof of Theorem~\ref{thm:ppa}]
The set of minimizers of $f,$ 
$$C=\left\{c\in X:f(c)=\inf_{x\in X} f(x)\right\}$$
is by the assumptions nonempty. Without loss of generality suppose $f_{\res_C}=0.$ We first show that the sequence $(x_n)$ is Fej\'er monotone with respect to $C.$ That is, for a given $c\in C,$ we are to verify
\begin{equation}
 \label{eq:fejerc}  d(x_n,c)\leq d(x_{n-1},c), \quad\text{for all } n\in\nat.
\end{equation}
Denote the $f(x_n)$-sublevel set of $f$ by
$$A=\left\{x\in X:f(x)\leq f(x_n)\right\}.$$
Then $c\in A,$ and $x_n=P_A(x_{n-1}),$ which together, along with Proposition~\ref{prop:proj}, yield $\alpha\left(x_{n-1},x_n,c\right)\geq\pi/2.$ Hence we get~\eqref{eq:fejerc}.

In the next step we show that the inequality
\begin{equation}
 \label{eq:estim}
 \lambda_k f(x_k)\leq \frac12 d(x_{k-1},c)^2 -\frac12 d(x_k,c)^2
\end{equation}
holds for any $c\in C$ and $k\in\nat.$ Indeed, from the definition of $x_k$ we have
$$\lambda_k f(x_k)+\frac12 d(x_k,x_{k-1})^2\leq \lambda_k f(p)+\frac12 d(p,x_{k-1})^2,$$
for any $p\in X.$ In particular, let $t\in[0,1)$ and $p_t=tx_k+(1-t)c,$ then
$$\frac12 d(x_k,x_{k-1})^2-\frac12 d(p_t,x_{k-1})^2\leq\lambda_k\left[ f(p_t)-f(x_k) \right].$$
Applying \eqref{eq:cat} to the above inequality gives
\begin{align*}
\lambda_k(1-t)\left[ f(c)-f(x_k) \right]\geq & -\frac{1-t}2 d\left(c,x_{k-1}\right)^2 \\ & + \frac{1-t}2 d\left(x_k,x_{k-1}\right)^2 \\ &+\frac{t(1-t)}2 d\left(x_k,c\right)^2,
\end{align*}
or, after taking into account that $f(c)=0,$ and $t\neq1,$
$$\lambda_k f(x_k)\leq \frac12 d\left(c,x_{k-1}\right)^2-\frac12 d\left(x_k,x_{k-1}\right)^2 -\frac{t}2 d\left(x_k,c\right)^2.$$
Since this inequality holds for all $t\in[0,1),$ we conclude that
$$\lambda_k f(x_k)\leq \frac12 d\left(c,x_{k-1}\right)^2-\frac12 d\left(x_k,x_{k-1}\right)^2 -\frac12 d\left(x_k,c\right)^2,$$
and therefore \eqref{eq:estim} holds true.

From \eqref{eq:estim} and from the monotonicity of $\left(f(x_n)\right)_n,$ we now obtain
$$2f(x_n)\sum_{k=1}^n\lambda_k \leq 2\sum_{k=1}^n\lambda_k f(x_k)\leq d(x_0,c)^2-d(x_n,c)^2,$$
and
\begin{equation} \label{eq:rate}
f(x_n)\leq\frac{d(x_0,c)^2}{\sum_{k=1}^n\lambda_k}.
\end{equation}

By the assumptions, the right hand side of the last inequality goes to $0,$ as $n\to\infty.$ We thence have that $(x_n)$ is a minimizing sequence, that is, $f(x_n)\to0$ as $n\to\infty.$

Assume now that a point $x_\infty\in X$ is a weak cluster point of $(x_n),$ that is, there exists a subsequence $\left(x_{n_k}\right)$ of $(x_n)$ such that $\left(x_{n_k}\right)\wto x_\infty.$ Since $f$ is lsc, and therefore weakly lsc by Lemma~\ref{lem:lsc}, we get $f\left(x_\infty\right)=0,$ and hence $x_\infty\in C.$ Applying Proposition~\ref{prop:fejer}\eqref{item:iii} finally gives $x_n\wto x_\infty,$ as $n\to\infty.$ This finishes the proof.
\end{proof}

The proof of Theorem~\ref{thm:flow} is similar to that of Theorem~\ref{thm:ppa}. The main ingredients come from \cite{mayer}.

\begin{proof}[Proof of Theorem \ref{thm:flow}]
Let $x\in \cldom f$ be a (starting) point. We first observe that $\left(T_\lambda x\right)_{\lambda\geq0}$ weakly converges to a point $z\in X,$ as $\lambda\to\infty,$ if and only if, for any sequence $(\lambda_n)\subset[0,\infty)$ such that $\lambda_n\to\infty,$ the sequence $\left(T_{\lambda_n} x\right)_n$ weakly converges to $z,$ as $n\to\infty.$ Take therefore an arbitrary sequence $(\lambda_n)\subset[0,\infty)$ such that $\lambda_n\to\infty.$ Denote
$$C=\left\{c\in X:f(c)=\inf_{x\in X} f(x)\right\}.$$
The set $C$ is nonempty, and $\left(T_{\lambda_n} x\right)_n$ is Fej\'er monotone with respect to $C$ by \cite[Theorem 2.38]{mayer}. Furthermore, the sequence $\left(T_{\lambda_n} x\right)_n$ is minimizing, that is,
$$\lim_{n\to\infty} f\left(T_{\lambda_n} x\right)= \inf_{y\in X} f(y),$$
by \cite[Theorem 2.39]{mayer}. The same argument as above finishes the proof: $f$ is weakly lsc (Lemma~\ref{lem:lsc}), and it follows that all weak cluster points of $\left(T_{\lambda_n} x\right)$ lie in $C.$ Finally, apply Proposition~\ref{prop:fejer}\eqref{item:iii}.
\end{proof}

\begin{proof}[Proof of Remark \ref{rem:uc}]
Here we will show that if the function $f$ is uniformly convex on bounded subsets of $\dom f$ with modulus $\phi,$ we have even strong convergence in Theorems~\ref{thm:ppa} and~\ref{thm:flow}. Indeed, let $(x_n)$ be the sequence from Theorem~\ref{thm:ppa}. We already know that it is bounded, and hence
\begin{equation} \label{eq:unif}
\frac14\phi\left(d\left(x_n,x_m\right)\right)\leq\frac12f(x_n)+\frac12f(x_m)-f\left(\frac{x_m+x_n}2\right),
\end{equation}
for any $m,n\in\nat.$ We also already know that $(x_n)$ is a minimizing sequence for $f,$ and since $\phi$ vanishes only at $0,$ we get by \eqref{eq:unif} that $(x_n)$ is Cauchy.

The case of Theorem~\ref{thm:flow} is similar.
\end{proof}

\bibliographystyle{siam}
\bibliography{ppa-cat}

\end{document}